\documentclass[12pt]{article}

\usepackage[letterpaper, margin=1.7cm]{geometry}

\usepackage{amsmath, amsfonts, amssymb, amsthm}
\usepackage{setspace}
\usepackage{commath}
\usepackage{mathtools}
\usepackage{bbm}
\usepackage{bm}
\usepackage{subfig} 
\usepackage{color}
\usepackage[normalem]{ulem}
\usepackage{kbordermatrix}
\usepackage{mathrsfs}
\usepackage{overpic}
\usepackage{cleveref}
\usepackage{enumerate}
\usepackage[singlelinecheck=false]{caption}

\definecolor{DarkGreen}{rgb}{0.2,0.6,0.2}

\def\red#1{\textcolor{black}{#1}}

\numberwithin{equation}{section}

\def\ignore#1{}

\def\<{\langle}\def\>{\rangle}
\def\Ind#1{{\mathbbmss 1}_{_{\scriptstyle #1}}}

\newtheorem{theorem}{Theorem}[section]
\newtheorem{proposition}[theorem]{Proposition}
\newtheorem{corollary}[theorem]{Corollary}
\newtheorem{lemma}[theorem]{Lemma}

\theoremstyle{definition}
 
\newtheorem{example}[theorem]{Example} 
\newtheorem{remark}[theorem]{Remark}

\begin{document}
\title{On the $p^{\text{th}}$ variation of a class of fractal functions}
\author{ Alexander Schied\thanks{Department of Statistics and Actuarial Science, University of Waterloo. E-mail: {\tt aschied@uwaterloo.ca}}
	 \and\setcounter{footnote}{6}
	Zhenyuan Zhang\thanks{
		Department of Statistics and Actuarial Science, University of Waterloo. E-mail: {\tt z569zhan@edu.uwaterloo.ca}
	 		\hfill\break The authors gratefully acknowledge financial support  from the
 Natural Sciences and Engineering Research Council of Canada through grant RGPIN-2017-04054}}
        
        \date{\normalsize First version: June 29, 2019\\
        this version: April 27, 2020    }
\vspace{-2cm}
\maketitle

\vspace{-1cm}

\begin{abstract}The concept of the $p^{\text{th}}$ variation of a continuous function $f$ along a refining sequence of partitions is the key to a pathwise It\^o integration theory with integrator $f$. Here, we analyze the $p^{\text{th}}$ variation of a class of fractal functions,  containing both the Takagi--van der Waerden and Weierstra\ss\ functions. We use a probabilistic argument to show that these functions have linear $p^{\text{th}}$ variation for a parameter $p\ge1$, which can be interpreted as the reciprocal Hurst parameter of the function. It is shown moreover that if functions are constructed from  (a skewed version of) the tent map, then the slope of the $p^{\text{th}}$ variation can be computed from the $p^{\text{th}}$ moment of a (non-symmetric) infinite Bernoulli convolution. Finally, we provide a recursive formula of these moments and use it to discuss the existence and non-existence of a signed version of the $p^{\text{th}}$ variation, which occurs in pathwise It\^o calculus when $p\ge3$ is an odd integer.
 \end{abstract}
 

\medskip


\section{Introduction}

Many random phenomena require a description by trajectories that are rougher (or smoother) than the sample paths of continuous semimartingales. A showcase example is the recent observation by Gatheral et al.~\cite{GatheralRosenbaum} that the realized  volatility of stocks and stock price indices is typically \lq\lq rough\rq\rq. To measure the degree of roughness of a function $f:[0,1]\to\mathbb{R}$, Gatheral et al.~\cite{GatheralRosenbaum} and others study expressions of the form
\begin{equation}\label{finite var sum}
\sum\big|f(t_{i+1})-f(t_i)\big|^p
\end{equation}
where the time points $t_i$ form a partition of $[0,1]$ and $p\ge1$ is a parameter. The intuition is that, if the mesh of the partition tends to zero, then there exists a number $q\in[1,\infty]$ such that the sums in~\eqref{finite var sum} will diverge for $p<q$ and converge to zero for $p>q$. This number $q$ can be regarded as the reciprocal of the Hurst parameter of $f$. Clearly, for a typical continuous semimartingale, the corresponding parameter $q$ would be equal to 2, whereas values larger than 6, or even larger than 10, are observed in~\cite{GatheralRosenbaum} for realized volatility trajectories. This observation has spawned a large amount of work on stochastic models for rough volatility. Typically, such models  rely on fractional Brownian motion or fractional Ornstein--Uhlenbeck processes in describing the rough volatility processes and on rough paths integration theory for the mathematical analysis.

Another recent development concerns a strictly pathwise  It\^o integration theory for \lq\lq rough\rq\rq\ integrators. For the case of quadratic variation (i.e., $q=2$, but simultaneously for  $q<2$), this theory goes back to F\"ollmer~\cite{FoellmerIto}. For $q>2$, such a theory was recently developed by Cont and Perkowski~\cite{ContPerkowski};  see also Gradinaru et al.~\cite{GradinaruEtAl} and Errami and Russo~\cite{ErramiRusso} for related earlier work.  It is technically  easier than standard rough paths calculus and can still be used to provide a path-by-path analysis of many situations in which  stochastic calculus is normally used; see, e.g.,~\cite{FoellmerSchiedBernoulli,SchiedSpeiserVoloshchenko} and the references therein for several case studies for the case of quadratic variation ($q=2$).
The use of pathwise It\^o integration also has the advantage that it is not dependent on probabilistic model assumptions and thus is inherently robust with respect to model risk.
This latter point is perhaps particularly important for  situations with $q>2$, because there  much fewer models are available than in the quadratic variation case. 

In this note, our goal is to establish that all functions in a well-studied class of fractal functions have  linear $p^{\text{th}}$ variation  on $[0,1]$, and thus to establish these functions as valid integrators in the pathwise It\^o calculus developed in~\cite{ContPerkowski}. This class of fractal functions contains  in particular  the classical Takagi--van der Waerden and Weierstra\ss\ functions. In doing so, we continue the work in~\cite{MishuraSchied2}, where  linear $p^{\text{th}}$ variation was established for  the special class of Takagi--Landsberg functions. One of our ultimate goals is to provide a class of possible models for \lq\lq rough\rq\rq\ trajectories that allow for an application of It\^o calculus and that are not bound by the restrictive assumption of Gaussianity. 

Our first main result, Theorem~\ref{thm}, shows first that a fractal function $f$ is either of bounded variation or have non-trivial  linear $q^{\text{th}}$ variation, where $q$ is computed from the parameterization of $f$. Moreover, the slope of the  $q^{\text{th}}$ variation is identified as the $q^{\text{th}}$ absolute moment of a certain random variable $Z$, which in case of the Takagi--van der Waerden functions with even $b$ has the law of an infinite Bernoulli convolution. Then we turn to a skewed version $\phi$ of the tent map and investigate a signed version of the $q^{\text{th}}$ variation, which would arise in pathwise It\^o calculus if $q\ge3$ is an  odd integer. In Theorem~\ref{signed var thm} we show that this signed $q^{\text{th}}$ variation may or may not exist, but that it will never vanish as long as $\phi$ is genuinely skewed. The proof of Theorem~\ref{signed var thm}  is based on some auxiliary results on the moments of a general non-symmetric infinite Bernoulli convolution, which may be of independent interest.

In Section~\ref{results section}, we present and prove our general results.  In Section~\ref{example section}, we discuss several explicit examples, which include the classical Weierstra\ss\ function and the Takagi--van der Waerden functions. Then we discuss the existence and nonexistence of the signed $p^{\text{th}}$ variation for a class of functions based on a skewed version of the tent map.

\section{General results}\label{results section}
We consider a base function $\phi:\mathbb R\to\mathbb R$ that is periodic with period 1, Lipschitz continuous, and vanishes on $\mathbb{Z}$. 
Our aim is to study the function
\begin{align}
f(t):=\sum_{m=0}^\infty \alpha^m\phi(b^mt),\qquad t\in[0,1],\label{vdw}
\end{align}
where  $b\in\{2,3,\dots\}$ and $\alpha\in (-1,1)$. We exclude the trivial case $\alpha=0$. In this case, the series on the right-hand side converges absolutely and uniformly in $t\in[0,1]$, so that $f$ is indeed a well defined continuous function. If 
$\phi(t)=\sin(2\pi t)$, then  $f$ is a Weierstra\ss\ function. For the tent map, $\phi(t)=\min_{z\in\mathbb{Z}}|t-z|$, the function $f$ belongs to the class of Takagi-van der Waerden functions. For instance, the classical Takagi function~\cite{Takagi} has the parameters $b=2$ and $\alpha=1/2$. Also the case of a general base function $\phi$ is well studied; see, e.g., the survey~\cite{BaranskiSurvey} and the references therein.

Here, we analyze the $p^{\text{th}}$ variation of the function $f$ along the sequence 
\begin{equation}\label{b-adic partitions}
\mathbb{T}_n:=\{kb^{-n}:k=0,\dots,b^n\},\quad n\in\mathbb{N},
\end{equation}
  of $b$-adic partitions of $[0,1]$ for $p\ge1$. Recall that a function $f\in C[0,1]$  admits the \emph{continuous $p^{\text{th}}$ variation $\<f\>_t^{(p)}$ along the sequence $(\mathbb T_n)$}, if for each $t\in[0,1]$,
  \begin{equation}\label{pth var eq}
  \<f\>_t^{(p)}:=\lim_{n\uparrow\infty}\sum_{k=0}^{\lfloor tb^n\rfloor}\big|f((k+1)b^{-n})-f(kb^{-n})\big|^p
  \end{equation}
  exists\footnote{In the sum on the right-hand side of \eqref{pth var eq}, the function $f$ will be formally evaluated at $1+b^{-n}>1$ if $t=1$ and $k=b^n$. To deal with this situation, we will assume here and in the sequel that all functions $f$ defined on $[0,1]$ will be extended to $[0,\infty)$ by putting $f(t):=f(\min\{t,1\})$.} and the function $t\mapsto \<f\>_t^{(p)}$ is continuous (see, e.g., Lemma 1.3 in~\cite{ContPerkowski}). According to F\"ollmer~\cite{FoellmerIto} in the case $p=2$ (and simultaneously for any $p<2$) and Cont and Perkowski~\cite{ContPerkowski} in the case of general even $p$ (and hence for any finite $p$), this notion of $p^{\text{th}}$ variation along a refining sequence of partitions is the key to  pathwise It\^o integration with integrator $f$. Note that  for $p>1$ the notion of $p^{\text{th}}$ variation is
different from the alternative concept of \emph{finite $p$-variation} defined in analogy to the total variation by means of a supremum taken over \emph{all} possible partitions of $[0,1]$ (see, e.g.,~\cite{FrizHairer}  or~\cite{LiuPromel}). 
Also, just as for the usual quadratic variation, the $p^{\text{th}}$ variation of any continuous function, and thus in particular of the sample paths of any continuous stochastic process, depends on the choice of the refining sequence of partitions  if $p>1$ (see also the discussion in Section 2 of~\cite{SchiedSpeiserVoloshchenko}).  For the special case $p=2$, certain robustness results are available that allow to translate results obtained for one refining sequence of partitions to another one that, in a certain sense, is comparable to the first; see~\cite{ContDas} and the references therein.

To state our first main result, we fix  $\phi$ and $b$ and  define the  coefficients
\begin{equation}\label{lambdas}
\lambda_{m,k}:=\frac{\phi\big((k+1)b^{-m}\big)-\phi\big(kb^{-m}\big)}{b^{-m}},\qquad m\in\mathbb{N},\, k=0,\dots, b^m-1.
\end{equation}
Next, let $(\Omega,\mathscr{F},\mathbb{P})$ be a probability space supporting an independent sequence $U_1,U_2,\dots$  of random variables with 
a uniform distribution on $\{0,1,\dots, b-1\}$ and define the stochastic processes
\begin{equation}\label{R and Y}
R_m:=\sum_{i=1}^mU_ib^{i-1}\quad\text{and}\quad Y_m:=\lambda_{m,R_m},\qquad m\in\mathbb{N}.
\end{equation}
Note that $R_m$ has a uniform distribution on $\{0,\dots, b^m-1\}$.

\begin{theorem}\label{thm} Under the assumptions stated above, the following assertions hold.
\begin{enumerate}
\item If $|\alpha|<1/b$, the function $f$ is of bounded variation.
\item If $|\alpha|=1/b$, then for all $t\in[0,1]$,
\begin{equation}\label{vanishing pvar}
\lim_{n\uparrow\infty}\sum_{k=0}^{\lfloor tb^n\rfloor}\big|f((k+1)b^{-n})-f(kb^{-n})\big|^p=0\qquad\text{for all $p>1$.}
\end{equation}

\item If $1/b<|\alpha|<1$, we let 
\begin{equation}\label{Z rv}
Z:=\sum_{m=1}^\infty(\alpha b)^{-m}Y_m.
\end{equation}
Then 
the function $f$ is  of bounded variation on $[0,1]$ if and only if $Z=0$ $\mathbb{P}$-a.s.  Otherwise, letting $q:=-\log_{|\alpha|}b$, we have for 
 all $t\in(0,1]$,
\begin{align}\label{pth}
\lim_{n\uparrow\infty}\sum_{k=0}^{\lfloor tb^n\rfloor}\big|f((k+1)b^{-n})-f(kb^{-n})\big|^p=\begin{cases}
0&\ \text{ if } p>q,\\
t\cdot \mathbb{E}[|Z|^{q}]&\ \text{ if } p=q,\\
+\infty &\ \text{ if } p<q.\end{cases}
\end{align}
\end{enumerate}
\end{theorem}

\goodbreak
\begin{remark}Let us comment on Theorem~\ref{thm}. 
\begin{enumerate}[{\rm(i)}] 
\item On the one hand, in the case  $|\alpha|=1/b$ the function $f$ may be nowhere differentiable and thus not be of bounded variation despite the fact that~\eqref{vanishing pvar} holds. An example is the classical Takagi function, which is obtained by taking $\phi$ as the tent map, $b=2$, and $\alpha=1/2$ (see, e.g.,~\cite{AllaartKawamura}).  
\item On the other hand, even if $1/b<|\alpha|<1$, it may happen that the function $f$ is not a fractal function. Indeed, the following example is given in~\cite{BaranskiSurvey}: let $\phi(t)=g(t)-\alpha g(bt)$ for some function $g$ on $\mathbb{R}$ that is periodic with period 1 and vanishes on $\mathbb{Z}$. Then $f=g$, and so $f$ will be smooth if $g$ is. 
\item If $\phi$ is the tent map, $b=2$, and $\alpha>0$, then $f$ is often called the Takagi--Landsberg  function with Hurst parameter $H=-\log_b\alpha$. The corresponding case of Theorem~\ref{thm} (c) is contained in Theorem 2.1 of~\cite{MishuraSchied2}\footnote{Note that in the printed version of~\cite{MishuraSchied2}, there is a factor $2^{1-1/H}$ missing in the statement of that theorem.}; see also Proposition~\ref{tent prop}. Note moreover that 
the $p^{\text{th}}$ variation of fractional Brownian motion with Hurst parameter $H\in(0,1)$ vanishes for $p>1/H$, is a nonconstant linear function of time for $p=1/H$, and is infinite for $p<1/H$ (see, e.g., 
\cite[Section 1.18]{Mishura}). Therefore,~\eqref{pth} 
suggests that also for general $\phi$, $b$, and $\alpha$, the number 
$$H:=-\frac1{\log_{|\alpha|}b}=-\log_b|\alpha|
$$
can be called the \emph{Hurst parameter} of $f$. This leads to the following alternative parameterization of the function $f$ in~\eqref{vdw} for $\alpha>0$,
\begin{equation}\label{f H eq}
f(t)=\sum_{m=0}^\infty b^{-H m}\phi(b^mt),\qquad t\in[0,1].
\end{equation}
See Figure~\ref{figure 1} for plots of $\<f\>_1^{(1/H)}$ as a function of $H$ for two different choices of $\phi$.
\end{enumerate}
\end{remark}

\begin{figure}
\begin{minipage}[b]{8cm}
\begin{overpic}[height=5cm]{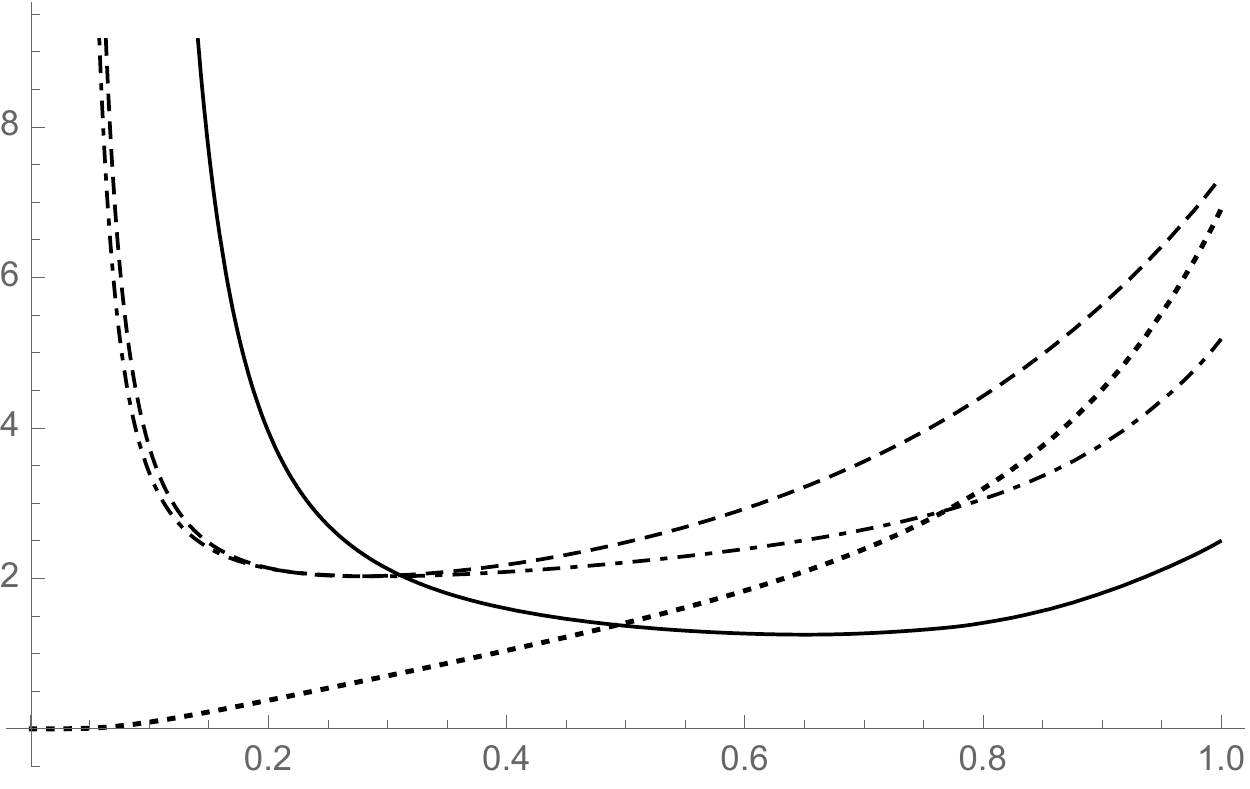}
\end{overpic}
\end{minipage}\quad
\begin{minipage}[b]{8cm}
\begin{overpic}[height=5cm]{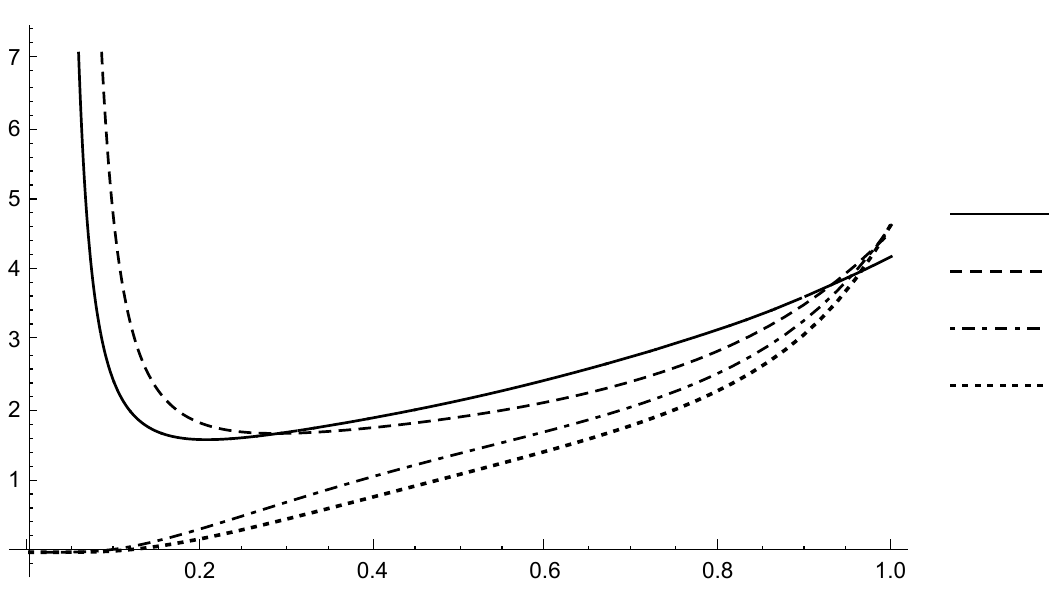}
\put(100,35.5){\small$b=2$}
\put(100,30.0){\small$b=3$}
\put(100,25){\small$b=4$}
\put(100,19){\small$b=5$}
\end{overpic}
\end{minipage}
\caption{The $(\frac1H)^{\text{th}}$ variation $\<f\>_1^{(1/H)}$ of the function $f$ in~\eqref{f H eq} as a function of $H\in(0,1)$, for several choices of $b$, and for $\phi(t)=5\min_{z\in\mathbb{Z}}|t-z|$ (left) and $\phi(t)=\frac12\sin(2\pi t)$ (right). Note that $\<f\>_1^{(1/H)}$ tends to 0 or $+\infty$ as $H\downarrow0$, depending on whether the $L^\infty$-norm of 
$Z_0:=\sum_{m=1}^\infty b^{-m}Y_m$
is less than or larger than~1.}\label{figure 1}
\end{figure}

Now we prepare for the proof of Theorem~\ref{thm}, which will be based on two auxiliary lemmas. To this end,
we fix $b\in\{2,3,\dots\}$ and let $f$ denote the function defined in~\eqref{vdw}. For $p\ge1$, $t\in[0,1]$,  $g\in C[0,1]$, and $n\in\mathbb{N}$, we define
\begin{equation}\label{Vptng}
V_{p,t,n}(g):=\sum_{k=0}^{\lfloor tb^n\rfloor}\big|g((k+1)b^{-n})-g(kb^{-n})\big|^p.
\end{equation}

\begin{lemma}\label{lemma 1}
For $n\in\mathbb{N}$, $p\ge1$ and $(Y_m)_{m\in\mathbb N}$ as in~\eqref{R and Y},
\[
V_{p,1,n}(f)=(|\alpha|^p b)^n \mathbb{E}\bigg[\Big|\sum_{m=1}^{n}(\alpha b)^{-m}Y_m\Big|^p\bigg].
\]
\end{lemma}

\begin{proof}
 Let the $n^{\text{th}}$ truncation of $f$ be given by
\[
f_n(t)=\sum_{m=0}^{n-1} \alpha^m\phi(b^mt).
\]
It is easy to see that $f_n(kb^{-n})=f(kb^{-n})$ for $k=0,\dots,b^n$ since $\phi$ vanishes on $\mathbb{Z}$. 
Hence, for $k=0,\dots,b^n-1$,
\begin{align*}
f((k+1)b^{-n})-f(kb^{-n})&=f_n((k+1)b^{-n})-f_n(kb^{-n})\\
&=\sum_{m=0}^{n-1}\alpha^m\Big(\phi((k+1)b^{m-n})-\phi(kb^{m-n})\Big)\\
&=\sum_{m=0}^{n-1}\alpha^mb^{m-n}\lambda_{n-m,k}.
\end{align*}
Using the fact that $R_m$ has a uniform distribution on $\{0,\dots, b^m-1\}$, we hence get
\begin{align}
V_{p,1,n}(f)&=\sum_{k=0}^{b^n-1}\Big|\sum_{m=0}^{n-1}\alpha^mb^{m-n}\lambda_{n-m,k}\Big|^p\nonumber\\
&=|\alpha|^{np}\sum_{k=0}^{b^n-1}\Big|\sum_{m=1}^{n}(\alpha b)^{-m}\lambda_{m,k}\Big|^p=|\alpha|^{np}b^n\mathbb{E}\bigg[\Big|\sum_{m=1}^{n}(\alpha b)^{-m}\lambda_{m,R_n}\Big|^p\bigg].\label{lemma 1 first eq}
\end{align}
Next note that for any $x$ and $n\geq m$, due to the periodicity of $\phi$ and by~\eqref{R and Y}, 
$$\phi\big(x+R_nb^{-m}\big)=\phi\Big(x+\sum_{i=1}^nU_ib^{i-1-m}\Big)=\phi\Big(x+\sum_{i=1}^mU_ib^{i-1-m}\Big)=\phi\big(x+R_mb^{-m}\big).
$$
 Hence, $\lambda_{m,R_n}=\lambda_{m,R_m}$, and so~\eqref{lemma 1 first eq} yields
$$V_{p,1,n}(f)=(|\alpha|^pb)^n\mathbb{E}\bigg[\Big|\sum_{m=1}^{n}(\alpha b)^{-m}Y_m\Big|^p\bigg].
$$
This completes the proof.
\end{proof}

The following simple lemma is a slightly strengthened version of~\cite[Lemma 3.1]{MishuraSchied2}.

\begin{lemma}\label{additive p-var lemma}For $p\ge1$ and $t\in(0,1]$, suppose  that $g\in C[0,1]$ is a function such that $V_{p,t,n}(g)\to0$ as $n\uparrow\infty$. Then, for $h\in C[0,1]$, the limit $\lim_n V_{p,t,n}(h)$ exists if and only if $\lim_n V_{p,t,n}(g+h)$ exists, and, in this case, both limits are equal.
\end{lemma}

\begin{proof}Minkowski's inequality yields that
\begin{align*}
\big(V_{p,t,n}(h)\big)^{1/p}-\big(V_{p,t,n}(g)\big)^{1/p}\le \big(V_{p,t,n}(g+h)\big)^{1/p}\le \big(V_{p,t,n}(g)\big)^{1/p}+\big(V_{p,t,n}(h)\big)^{1/p}.
\end{align*}
Passing to the limit $n\uparrow\infty$ thus yields the \lq\lq only if\rq\rq\ part of the assertion. Rearranging the preceding inequality or replacing $h$ with $g+h$ and $g$ with $-g$ yields the \lq\lq if\rq\rq\ part.
\end{proof}

\begin{proof}[Proof of Theorem~\ref{thm}] Let $C$ denote a Lipschitz constant for $\phi$. Then, by~\eqref{lambdas}, $|\lambda_{m,k}|\le C$ and in turn $|Y_m|\le C$. 

{\bf (a) $\bm{|\alpha|<1/b}$.} Taking $p=1$, Lemma~\ref{lemma 1} yields that 
\begin{align*}0\le V_{1,1,n}(f)&
=\mathbb{E}\bigg[\bigg|\sum_{m=0}^{n-1}\alpha^mb^{m}Y_{n-m}\bigg|\bigg]\le C\sum_{m=0}^\infty(|\alpha| b)^m<\infty.
\end{align*}
Thus, the sequence $(V_{1,1,n}(f))$ is bounded uniformly in $n$. Next, for $p=1$, the triangle inequality yields that $V_{1,1,n}(f)\le V_{1,1,n+1}(f)$ for all $n$, and so $V_{1,1,n}(f)$ converges to a finite limit as $n\uparrow\infty$. But since $f$ is continuous, this limit must coincide with the total variation of $f$ (see, e.g., Theorem 2 in \S5 of Chapter VIII in~\cite{Natanson}). 

{\bf (b) $\bm{|\alpha|=1/b}$.} Here we show that $V_{p,1,n}(f)\to0$ if $p>1$.  Indeed, in this case, Lemma~\ref{lemma 1} yields that 
\begin{align*}
0\le V_{p,1,n}(f)^{1/p}&=(|\alpha|^pb)^{n/p}\mathbb{E}\bigg[\Big|\sum_{m=1}^{n}(\text{sgn}\,\alpha)^{m}Y_m\Big|^p\bigg]^{1/p}\le (|\alpha|b^{1/p})^n\sum_{m=1}^n\|Y_m\|_{L^p}\le (|\alpha|b^{1/p})^nnC,
\end{align*}
and the rightmost term tends to zero as $n\uparrow\infty$. Since $0\le V_{p,t,n}(f)\le V_{p,1,n}(f)$, the result holds for all $t\in[0,1]$. 

{\bf (c) $\bm{|\alpha|>1/b}$.} First  we deal with the case $Z=0$ $\mathbb{P}$-a.s.  Then 
$$\sum_{m=1}^{n}(\alpha b)^{-m}Y_m=-\sum_{m=n+1}^{\infty}(\alpha b)^{-m}Y_m\quad\text{$\mathbb{P}$-a.s.}
$$
Hence, Lemma~\ref{lemma 1} yields that for $p=1$, 
\begin{align*}
0&\le V_{p,1,n}(f)= (|\alpha|b)^n\mathbb{E}\bigg[\bigg|\sum_{m=n+1}^{\infty}(\alpha b)^{-m}Y_m\bigg|\bigg]\le  C(|\alpha|b)^n\bigg|\sum_{m=n+1}^\infty(\alpha b)^{-m}\bigg|
=\frac{C}{|\alpha b-1|}.\end{align*}
Thus, we conclude as in (a) that $f$ must be of bounded variation. Once~\eqref{pth} will have been established,  the converse implication will follow by taking $p:=1<q$ in~\eqref{pth}.

For the remainder of the proof,  we suppose that $\mathbb{P}[Z\neq0]>0$, which implies $\mathbb{E}[|Z|^p]>0$ for any $p\geq 1$.   The fact that $|Y_m|\le C$ implies that  $\sum_{m=1}^{n}(\alpha b)^{-m}Y_m$  converges boundedly to $Z$ as $n\uparrow\infty$.  Therefore, for any $p\ge1$,
$$\mathbb{E}\bigg[\Big|\sum_{m=1}^{n}(\alpha b)^{-m}Y_m\Big|^p\bigg]\longrightarrow\mathbb{E}[|Z|^p]<\infty\qquad\text{as $n\uparrow\infty$.}
$$
For  $p>q$  we have $|\alpha|^pb<1$, and so Lemma~\ref{lemma 1} yields $V_{p,1,n}(f)\to0$. This establishes the first case in~\eqref{pth} for $t=1$. Since $0\le V_{p,t,n}(f)\le V_{p,1,n}(f)$, the result holds for all $t\in[0,1]$. 

For $p=q$ we get $V_{p,1,n}(f)\to \mathbb{E}[|Z|^p]>0$, and this yields the second case in~\eqref{pth} for $t=1$.  To establish the assertion also for $t\in(0,1)$, 
we observe that~\eqref{vdw} implies that
$$f(t)=\phi(t)+\alpha f(bt),\qquad t\in[0,1/b].
$$
Since $\phi$ is Lipschitz continuous and $p>1$, one easily gets $V_{p,t,n}(\phi)\to0$ for all $t$. Moreover, $V_{p,1/b,n}(f(b\,\cdot))=V_{p,1,n-1}(f)\to \mathbb{E}[|Z|^p]$. Hence, Lemma~\ref{additive p-var lemma} 
yields that
$$\lim_{n\uparrow\infty}V_{p,1/b,n}(f)=|\alpha|^p\lim_{n\uparrow\infty}V_{p,1/b,n}(f(b\,\cdot))=|\alpha|^p \mathbb{E}[|Z|^p]=\frac1b \mathbb{E}[|Z|^p].
$$ 
Iterating this argument gives 
\begin{equation}\label{bnu eq}
\lim_nV_{p,b^{-\nu},n}(f)=b^{-\nu}\mathbb{E}[|Z|^p]\qquad\text{for all $\nu\in\mathbb{N}$.}
\end{equation}  Next, for $\nu\in\mathbb{N}$, $k\in\{0,..,b^\nu-1\}$, and $t\in[ kb^{-\nu},1]$, the periodicity of $\phi$ implies that
\begin{align*}
f(t-kb^{-\nu})&=\sum_{m=0}^{\nu-1}\alpha^m\big(\phi(b^m(t-kb^{-\nu}))-\phi(b^mt)\big)+f(t)=:g(t)+ f(t).
\end{align*}
Since $g$ is Lipschitz continuous, it follows from Lemma~\ref{additive p-var lemma} and~\eqref{bnu eq} that 
$$V_{p,(k+1)b^{-\nu},n}(f)-V_{p,kb^{-\nu},n}(f)\longrightarrow b^{-\nu}\mathbb{E}[|Z|^p]\qquad\text{as $n\uparrow\infty$.}$$
Thus, we get $V_{p,t,n}(f)\to t\mathbb{E}[|Z|^p]$ whenever $t=kb^{-\nu}$ for certain $\nu\in\mathbb{N}$ and $k\in\{0,\dots, b^\nu\}$.  A sandwich argument then extends this fact to all $t\in[0,1]$.

Finally, if $p<q$, then $|\alpha|^pb>1$, and so $V_{p,1,n}(f)\to\infty$ by Lemma~\ref{lemma 1}. The analogous fact for $0<t<1$ can be proved as in the final part of the proof of Theorem 2.1 in~\cite{MishuraSchied2}.
\end{proof}

Combining Theorem~\ref{thm} with~\cite[Theorem 4.1]{HuLau} and~\cite[Theorem 2.4]{BaranskiSurvey} yields immediately the following corollary. See~\cite{BaranskiSurvey,HuLau} for further conditions that are equivalent to its statements (d)--(f).

\goodbreak
\begin{corollary}\label{cor}In the context of Theorem~\ref{thm} (c), the following conditions are equivalent.
\begin{enumerate}
\item $f$ is not of bounded variation.
\item For $q=-\log_{|\alpha|}b$, the $q^{\text{th}}$ variation of $f$ is strictly positive.
\item The random variable 
$Z=\sum_{m=1}^\infty (\alpha b)^{-m}Y_m$
satisfies $\mathbb P[Z\neq0]>0$.
\end{enumerate}
If, in addition,  $\phi$ is piecewise $C^1$, $\alpha\in(1/b,1)$, and $H:=-\log_b\alpha$, then {\rm(a)---(c)} are equivalent to each of the following conditions.  
\begin{enumerate}
\addtocounter{enumi}{3}
\item  $f$ is not piecewise $C^1$.
\item $f$ is nowhere differentiable.
\item The box dimension of the graph of $f$ is $2-H$.
\end{enumerate}
\end{corollary}

Now we state a sufficient condition for the  properties {\rm(a)---(c)} in Corollary~\ref{cor}. This condition is easy to verify and obviously satisfied  for the  Takagi--van der Waerden and Weierstra\ss\ functions. We refer to Propositions 4.3 and 5.6 in~\cite{HuLau} and Theorem 5 in~\cite{PrzytyckiUrbanski} for  other sufficient conditions.

\begin{proposition}\label{condition prop} Suppose that $\alpha\in(1/b,1)$ and that 
\begin{equation}\label{cond}
\{0\}\neq \{\phi(b^{-k}):k\in\mathbb{N}\}\subset[0,\infty).
\end{equation}
 Then the properties {\rm(a)---(c)} in Corollary~\ref{cor} are satisfied.
 \end{proposition}

\begin{proof} We show that condition (c) of Corollary~\ref{cor} is satisfied given~\eqref{cond}. To this end, let $C$ denote the Lipschitz constant of $\phi$. By assumption, there is $M\in\mathbb{N}$ such that $\phi(b^{-M})>0$. Choose $N>M$ and $\delta>0$ such that
\[
C \sum_{m=N}^\infty (\alpha b)^{-m} <\phi(b^{-M})-\delta.
\]
Then, for $\omega\in\{U_1=0,U_2=0,\dots,U_N=0\}$ and $m\le N$, we have $Y_m(\omega)=\lambda_{m,0}=b^m\phi(b^{-m})\geq 0$ and hence
\begin{align*}
\sum_{m=1}^{N-1}(\alpha b)^{-m}Y_m(\omega)\ge (\alpha b)^{-M}Y_M(\omega)\ge \phi(b^{-M}).
\end{align*}
Therefore,
\begin{align*}
|Z(\omega)|&\geq \bigg|\sum_{m=1}^{N-1}(\alpha b)^{-m} Y_m(\omega)\bigg|-\bigg|\sum_{m=N}^{\infty}(\alpha b)^{-m} Y_m(\omega)\bigg|\ge \phi(b^{-M})-C \sum_{m=N}^\infty (\alpha b)^{-m} >\delta.
\end{align*}
Since $\mathbb{P}[U_1=0,U_2=0,\dots,U_N=0]=b^{-N}>0$, we cannot have $\mathbb P[Z=0]=1$.
\end{proof}

 \begin{remark}By considering $\widetilde\phi(t):=-\phi(t)$ or $\widehat\phi(t):=\phi(-t)$ or $\overline\phi(t):=-\phi(-t)$, one sees that~\eqref{cond} can be replaced by  several similar conditions. For instance, requiring~\eqref{cond}  for $\overline\phi$ is equivalent to the condition   $\{0\}\neq \{\phi(1-b^{-k}):k\in\mathbb{N}\}\subset(-\infty,0]$. Another easy consequence is that, for any base function $\phi$ which has a nonvanishing right derivative  at $0$ or a nonvanishing left derivative at $1$, there exists some $b\in\mathbb{N}$ such that the properties {\rm(a)---(c)} in Corollary~\ref{cor} are satisfied.
 \end{remark}
 
 \section{Examples and signed $p^{\text{th}}$ variation}\label{example section}

Proposition~\ref{condition prop} and Theorem~\ref{thm} yield immediately the following corollary.

\begin{corollary}
For a fixed $\alpha\in(1/b,1)$ and $q=-\log_\alpha b$, there exists a constant $K\in(0,\infty)$ such that the Weierstra\ss\ function,
\[
w(t)=\sum_{n=0}^\infty \alpha^n \sin(2\pi b^nt),\quad t\in[0,1],
\]
has linear  $q^{\text{th}}$ variation, $\<w\>_t^{(q)}=tK$,  along the sequence of $b$-adic partitions~\eqref{b-adic partitions}.\end{corollary}

Now we turn to the class of Takagi--van der Waerden functions, which corresponds to the case in which $\phi$ is  the tent map,
\begin{equation}\label{tent map}
\phi(t)=\min_{z\in\mathbb{Z}}|t-z|,\qquad t\in\mathbb{R}.
\end{equation}
For $b=2$, the following result is contained in~\cite[Theorem 2.1]{MishuraSchied2}.\footnote{Note that in the printed version of~\cite{MishuraSchied2}, there is a factor $2^{1-1/H}$ missing in the statement of that theorem.} It characterizes the law of $Z$ in terms  of an infinite Bernoulli convolution. Recall that the law of a random variable $\widetilde Z$ is a (symmetric) infinite Bernoulli convolution with parameter $\beta\in(-1,1)$  if there is an i.i.d.~sequence $(\widetilde Y_m)_{m=0,1,\dots}$ of $\{-1,+1\}$-valued random variables with a symmetric Bernoulli distribution such that 
$$\widetilde Z=\sum_{m=0}^\infty\beta^m\widetilde Y_m.
$$
On the other hand, part (b) of the following proposition yields in particular that if $b$ is odd,  the random variables $(Y_m)_{m\in\mathbb{N}}$ are no longer independent.

\begin{proposition}\label{tent prop}Let $\phi$ be as in~\eqref{tent map} and $1/b<|\alpha|<1$.
\begin{enumerate}
\item If $b$ is even, the random variables $(Y_m)_{m\in\mathbb{N}}$ form an i.i.d.~sequence of symmetric $\{-1,+1\}$-valued infinite Bernoulli random variables. In particular, for $Z$ as  in~\eqref{Z rv}, the law of  $\widetilde Z:=\alpha bZ$  is the (symmetric) infinite Bernoulli convolution with parameter $1/(\red{|\alpha |}b)$, and for $q=-\log_{|\alpha|} b$ we have $\<f\>_t^{(q)}=t(\red{|\alpha|} b)^{-q}\mathbb{E}[|\widetilde Z|^{q}]$.
\item If $b$ is odd, the $(Y_m)_{m\in\mathbb{N}}$ form a Markov chain on $\{-1,0,+1\}$ with initial distribution $\left[\frac{b-1}{2b},\frac{1}{b},\frac{b-1}{2b}\right]$  and transition matrix
\[
  P= \kbordermatrix{
    & -1 &0 & +1  \\
    -1 & \frac{b+1}{2b} & 0 & \frac{b-1}{2b} \\
    0& \frac{b-1}{2b} &\frac{1}{b}& \frac{b-1}{2b}  \\
    +1& \frac{b-1}{2b}& 0 & \frac{b+1}{2b} 
  }.
\]
\end{enumerate} \end{proposition}

\begin{proof}
{\bf (a) $b$ is even.} This is a special case of Proposition~\ref{skew prop}.

{\bf (b) $b$ is odd.} The initial distribution of $Y_1$ is obvious. Next, we observe that $U_1,\dots, U_{m}$ can be recovered from $R_{m}$ so that
\begin{equation}\label{sigma fields}
\sigma(Y_1,\dots,Y_m)\subseteq\sigma(U_1,\dots, U_{m})=\sigma(R_{m})\qquad \text{for $m=1,2,\dots$.}
\end{equation}
Now we consider the event $\{Y_m=0\}$ and note that it coincides with $\{R_m=\frac{b^m-1}2\}$. But the latter event is equal to $\{U_1=\cdots=U_m=\frac{b-1}2\}$, which is in turn contained in  $\{R_{m-1}=\frac{b^{m-1}-1}2\}=\{Y_{m-1}=0\}$. It follows that
\begin{align*}
\mathbb P[Y_m=0|R_{m-1}]=\mathbb P\Big[U_m=\frac{b-1}2\,\Big]\Ind{\{Y_{m-1}=0\}}=\frac1b\Ind{\{Y_{m-1}=0\}}.
\end{align*}
In view of~\eqref{sigma fields}, this establishes the Markov property for the event $\{Y_m=0\}$ and gives the second column of the transition matrix $P$.

Next, we have  by~\eqref{lambdas} and~\eqref{R and Y}
that 
\[\{Y_m=1\}=\left\{R_mb^{-m}<\left(\frac{b^m-1}{2}\right)b^{-m}\right\}=\left\{R_m<\frac{b^m-1}{2}\right\}.
\]
Thus, the independence of $U_m$ and $R_{m-1}$ yields that for $k\in\{0,\dots, b^{m-1}-1\}$,
\begin{align}\label{cond prob eq}
\mathbb P[Y_m=1|R_{m-1}=k]&=\mathbb P\bigg[U_m<\frac b2-\Big(k+\frac12\Big)b^{1-m}\bigg].
\end{align}
If $0\le k<\frac12(b^{m-1}-1)$, which corresponds to $Y_{m-1}=1$, then
\begin{equation}\label{kb 1 eq}
\frac {b-1}2<\frac b2-\Big(k+\frac12\Big)b^{1-m}\le \frac b2.
\end{equation}
Since there is no integer in the interval $(\frac{b-1}2,\frac b2]$, we see from~\eqref{cond prob eq} that, whenever $0\le k<\frac12(b^{m-1}-1)$, the probability $\mathbb P[Y_m=1|R_{m-1}=k]$ is independent of $k$ and equal to 
$$\mathbb P\Big[U_m\le\frac12(b-1)\Big]=\frac{b+1}{2b}.
$$
Likewise, for $\frac12(b^{m-1}-1)\le k\le b^{m-1}-1$,  which corresponds to $Y_m=0$ or $Y_m=-1$, we get
\begin{equation}\label{kb 2 eq}
\frac b2-1+\frac1{2b^{m-1}}\le\frac b2-\Big(k+\frac12\Big)b^{1-m}\le\frac {b-1}2. 
\end{equation}
Again, there is no integer in the interval $[\frac b2-1+\frac1{2b^{m-1}},\frac {b-1}2)$, and so~\eqref{cond prob eq} implies that
$$\mathbb P[Y_m=1|R_{m-1}=k]=\mathbb P\bigg[U_m<\frac {b-1}2\bigg]=\frac{b-1}{2b}.
$$
Altogether, we have shown that 
\begin{align*}
\mathbb P[Y_m=1|R_{m-1}]=\frac{b+1}{2b}\Ind{\{Y_{m-1}=1\}}+ \frac{b-1}{2b}\Ind{\{Y_{m-1}=0\}\cup\{Y_{m-1}=-1\}}.
\end{align*}
In view of~\eqref{sigma fields}, this establishes the Markov property for the event $\{Y_m=1\}$ and gives the first column of the transition matrix $P$. The analogous result for $\{Y_m=-1\}$ follows by a symmetry argument.
\end{proof}

In Remark 1.7 of~\cite{ContPerkowski} it is conjectured that, if $p$ is an odd integer, the following \emph{signed  $p^{\text{th}}$ variation} of $f$,
\begin{equation}\label{signed pth variation}
\lim_{n\uparrow\infty}\sum_{k=0}^{\lfloor tb^n\rfloor}\big(f((k+1)b^{-n})-f(kb^{-n})\big)^p,\qquad t\in[0,1],
\end{equation}
 will typically vanish for all $t\in[0,1]$.
To discuss this conjecture, we will now study the fractal functions $f$ arising from the following 
skewed version of the tent map,
\begin{equation}\label{skewed tent map}
\phi(t):=\begin{cases}t\frac b{2\ell}&\text{if $0\le t\le \ell/b$,}\\
(1-t)\frac{b}{2(b-\ell)}&\text{if $\ell/b\le t\le1$,}
\end{cases}
\end{equation}
where  $b\in\{2,3,\dots\}$ and $\ell\in\{1,\dots,b-1\}$ are fixed. Then we extend $\phi$ to all of $\mathbb R$ by periodicity. Note that if $b$ is even and $\ell=b/2$, then $\phi$ is equal to the standard tent map~\eqref{tent map}, and so the following Proposition~\ref{skew prop} contains Proposition~\ref{tent prop} (a) as a special case. See Figure~\ref{skewed tent map fig} for  plots of two fractal functions $f$   corresponding to specific choices of $b$ and $\ell$ in~\eqref{skewed tent map}.

\begin{figure}[h]
\begin{minipage}[b]{8cm}
\begin{overpic}[height=5cm]{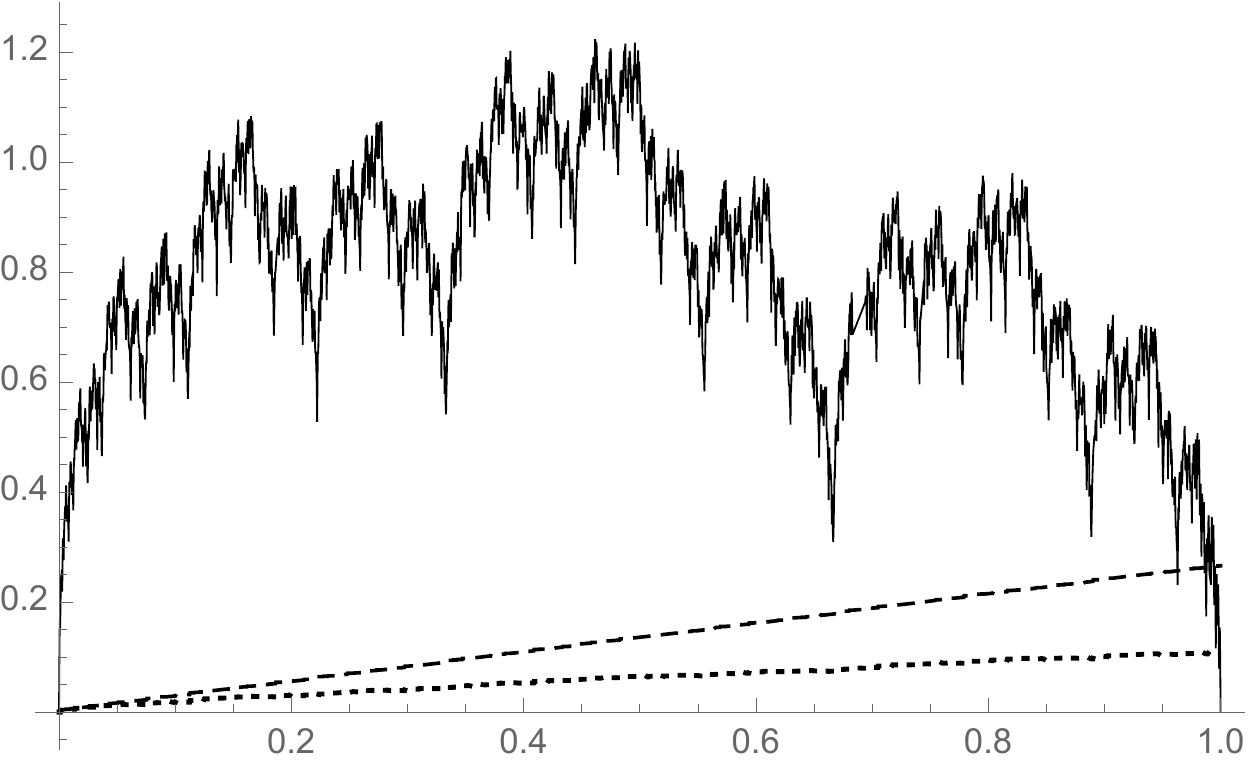}
\end{overpic}
\end{minipage}\qquad\qquad
\begin{minipage}[b]{8cm}
\begin{overpic}[height=5cm]{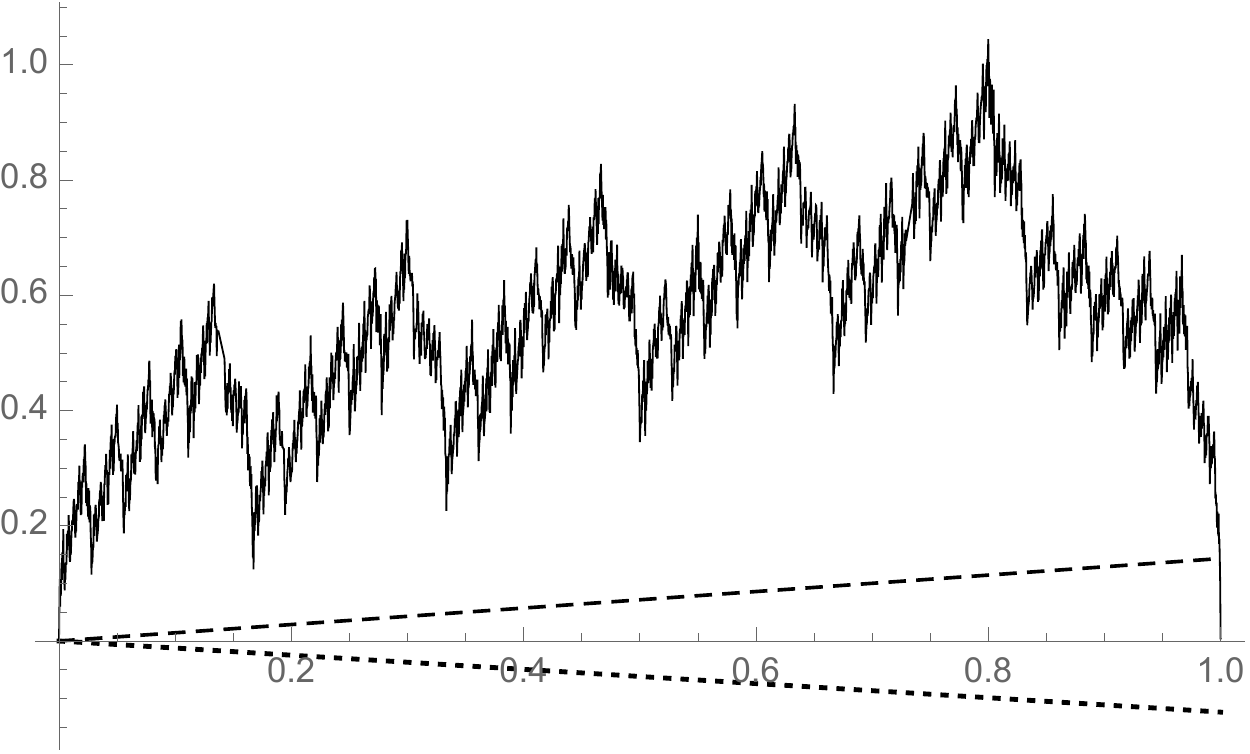}
\end{overpic}
\end{minipage}\caption{Functions $f$ for the skewed tent map~\eqref{skewed tent map} for $b=3$, $\ell=1$ (left) and $b=6$, $\ell=5$ (right) with $\alpha=b^{-1/3}$ in both cases. The dashed and dotted lines, respectively, are the  functions $$t\longmapsto \sum_{k=0}^{\lfloor tb^6\rfloor}\big|f((k+1)b^{-6})-f(kb^{-6})\big|^3\quad\text{and}\quad t\longmapsto \sum_{k=0}^{\lfloor tb^6\rfloor}\big(f((k+1)b^{-6})-f(kb^{-6})\big)^3.$$}
\label{skewed tent map fig}
\end{figure}

\begin{proposition}\label{skew prop} Let $\phi$ be as in~\eqref{skewed tent map} for given $b\in\{2,3,\dots\}$ and $\ell\in\{1,\dots,b-1\}$. Define
\begin{equation}\label{munu}
\mu:=\frac{-b}{2(b-\ell)}\quad\text{and}\quad \nu:=\frac{b}{2\ell}.
\end{equation}
Then $(Y_m)_{m\in\mathbb N}$ is an i.i.d.~sequence of $\{ \mu,\nu\}$-valued random variables with  $\mathbb P[Y_m=\nu]=\ell/b$ and $\mathbb E[Y_m]=0$. 
\end{proposition}

\begin{proof}Fix $m\in\mathbb{N}$ and define the function
$$\psi(x):=b^{m}\big(\phi(x+b^{-m})-\phi(x)\big).
$$
For $x\in\{kb^{-m}:k\in\mathbb{Z}\}$, we have $\psi(x)\in\{-\frac b{2(b-\ell)},\frac b{2\ell}\}$. More precisely, we have $\psi(kb^{-m})=\frac b{2\ell}=\nu$ for $k=0,\dots,\ell b^{m-1}-1$ and $\psi(kb^{-m})= -\frac b{2(b-\ell)}=\mu$ for $k=\ell b^{m-1},\dots, b^m-1$. 
 Moreover, $R_m\in\{0,\dots, \ell b^{m-1}-1\}$ if and only if $U_m \in\{0,\dots,\ell-1\}$. It follows that 
$$\mathbb{P}[Y_m=\nu\,|\,R_{m-1}]=\mathbb{P}[U_m\in\{0,\dots,\ell-1\}\,|\,R_{m-1}]=\mathbb{P}[U_m\in\{0,\dots,\ell-1\}]=\frac\ell b.$$
In view of~\eqref{sigma fields}, this proves that $Y_m$ is independent of $Y_1,\dots, Y_{m-1}$ and has the claimed  distribution. 
\end{proof}

With the preceding proposition, we are able to prove  the following result on the signed $p^{\text{th}}$ variation~\eqref{signed pth variation} of the functions $f$ arising from the skewed tent map~\eqref{skewed tent map}.

\begin{theorem}\label{signed var thm}In the context of Proposition~\ref{skew prop}, suppose that $\alpha\in(-1,1)$ is such that $|\alpha|>1/b$ and $q=-\log_{|\alpha|}b$ is an odd integer. Then:
\begin{enumerate}
\item If $b$ is even and $\ell=b/2$, the signed $q^{\text{th}}$ variation~\eqref{signed pth variation} exists and vanishes identically.
\item If $\alpha>0$, then the signed $q^{\text{th}}$ variation of $f$ exists and is given by 
\begin{equation}\label{signed qth variation}
\lim_{n\uparrow\infty}\sum_{k=0}^{\lfloor tb^n\rfloor}\big(f((k+1)b^{-n})-f(kb^{-n})\big)^q=t\cdot\mathbb E\big[Z^q\big],\qquad t\in[0,1],
\end{equation}
where $Z$ is as in Theorem~\ref{thm}. Moreover, $\mathbb E[Z^q]$
 is strictly positive if $\ell<b/2$ and strictly negative if $\ell>b/2$.
 \item If $\alpha<0$ and $\ell\neq b/2$, then 
 \begin{equation}\label{nonexist signed qth variation}
\lim_{n\uparrow\infty}\,(-1)^n\sum_{k=0}^{\lfloor tb^n\rfloor}\big(f((k+1)b^{-n})-f(kb^{-n})\big)^q=t\cdot\mathbb E\big[Z^q\big],\qquad t\in[0,1].
\end{equation}
In particular, the signed $q^{\text{th}}$ variation of $f$ exists only along $(\mathbb T_{2n})_{n\in\mathbb N}$ or along $(\mathbb T_{2n+1})_{n\in\mathbb N}$. 
\end{enumerate}
\end{theorem}

The proof of the preceding theorem is based on the following auxiliary results on the moments of general non-symmetric infinite Bernoulli convolutions.  The first is a recursive formula for the moments of a general non-symmetric infinite Bernoulli convolution. A formula for the moments of a standard (symmetric) infinite Bernoulli convolution was given in~\cite{EscribanoEtAl}.

\begin{lemma}\label{Bernoulli moment lemma} Suppose that $\mu,\nu\in\mathbb R$, $p\in(0,1)$, and $(\widetilde Y_m)_{m\in\mathbb N}$ is an i.i.d.~sequence of  $\{\mu,\nu\}$-valued random variables with $\mathbb P[\widetilde Y_m=\nu]=p$. For $\gamma\in(-1,1)$, let $\widetilde Z$ be the random variable
\begin{equation}
\widetilde Z=\sum_{m=1}^\infty\gamma^m\widetilde Y_m.
\end{equation}
Then, for $k\in\mathbb N$, the $k^{\text{th}}$ moment of $\widetilde Z$ is given by the following recursive formula,
\begin{equation}\label{Bernoulli lemma eq}
\mathbb E[\widetilde Z^k]=\frac{\gamma^k}{1-\gamma^k}\sum_{j=0}^{k-1}{k\choose j}\Big(p\nu^{k-j}+(1-p)\mu^{k-j}\Big)\mathbb E[\widetilde Z^j].
\end{equation}
\end{lemma}

\begin{proof} Note that $\widetilde Z$ has the same law as $\gamma(\widetilde Z+\widetilde Y_{0})$, where $\widetilde Y_{0}$ is such that $(\widetilde Y_m)_{m=0,1,\dots}$ is an i.i.d.~sequence. Conditioning on $\widetilde Y_{0}$ hence yields that for  $k\in\mathbb N$,
\begin{align*}
\mathbb E[\widetilde Z^k]&=p\mathbb E[(\gamma \widetilde Z+\gamma \nu)^k]+(1-p)\mathbb E[(\gamma \widetilde Z+\gamma \mu)^k]\\
&=\gamma^k\mathbb E\bigg[\sum_{j=0}^k{k\choose j}p\widetilde Z^j\nu^{k-j}+\sum_{j=0}^k{k\choose j}(1-p)\widetilde Z^j\mu^{k-j}\bigg]\\
&=\gamma^k\mathbb E[\widetilde Z^k]+ \gamma^k\sum_{j=0}^{k-1}{k\choose j}\Big(p\nu^{k-j}+(1-p)\mu^{k-j}\Big)\mathbb E[\widetilde Z^j].
\end{align*}
This yields~\eqref{Bernoulli lemma eq}.
\end{proof}

\begin{example}In Theorem~\ref{signed var thm}, suppose that $\alpha=b^{-1/3}$, so that $q=3$, and then let $\gamma=1/(\alpha b)$, $p=\ell/b$, and $\mu$ and $\nu$ as in~\eqref{munu}. Our formula~\eqref{Bernoulli lemma eq} gives for $k=2$ that $\mathbb E[ Z^2]=\gamma^2/(1-\gamma^2)$. Hence, for $k=3$ we get 
$$\mathbb E[ Z^3]=\frac{b^3(b-2\ell)}{8(b^2-1)\ell^2(b-\ell)^2}.
$$
When taking $b=3$ and $\ell=1$ as in the left-hand panel of Figure~\ref{skewed tent map fig}, we get $\mathbb E[ Z^3]=27/256$. For $b=6$ and $\ell=5$ as in the right-hand panel of Figure~\ref{skewed tent map fig}, we get $\mathbb E[ Z^3]=-875/6912\approx -0.1266$.
\end{example}

In the context of Theorem~\ref{signed var thm}, the random variables $Y_m$ are centered. It turns out that in this situation the odd moments of $Z$ have a common sign as long as $\alpha>0$. This is the content of the following lemma. 

\begin{lemma}\label{Bernoulli lemma}In the setting of  Lemma~\ref{Bernoulli moment lemma}, suppose in addition that $\mu<0<\nu$, $\mathbb E[\widetilde Y_m]=0$,  and $\gamma\in(0,1)$.  Then, for any given odd number $k\ge3$,
\begin{enumerate}
\item $\mathbb E[\widetilde Z^k]=0$ if and only if $\nu=-\mu$;
\item $\mathbb E[\widetilde Z^k]>0$ if and only if $\nu>-\mu$;
\item $\mathbb E[\widetilde Z^k]<0$ if and only if $\nu<-\mu$.
\end{enumerate}
\end{lemma}

\begin{proof}
If $\nu=-\mu$, then $\widetilde Z$ has a symmetric distribution and hence $\mathbb E[\widetilde Z^k]=0$. Thus it suffices to establish the implication \lq\lq$\nu>-\mu\Rightarrow \mathbb E[\widetilde Z^k]>0$\rq\rq, because the corresponding implication in (c) then follows by considering the random variables $-\widetilde Y_m$. So let us assume that $\nu>-\mu$. The fact that the $\widetilde Y_m$ are centered allows us to assume without loss of generality that $\nu=1/p$ and $-\mu=1/(1-p)$, for otherwise we multiply all random variables with $1/(\nu p)$. Then~\eqref{Bernoulli lemma eq} becomes
\begin{equation}\label{special Bernoulli lemma eq} 
\mathbb E[\widetilde Z^k]=\frac{\gamma^k}{1-\gamma^k}\sum_{j=0}^{k-1}{k\choose j}\Big(p^{1+j-k}+(-1)^{k-j}(1-p)^{1+j-k}\Big)\mathbb E[\widetilde Z^j].
\end{equation}
Our assumption $\nu>-\mu$ implies that $0<p<1/2$  so that  $p^{1+j-k}>(1-p)^{1+j-k}$ for $j=0,\dots,k-2$. Moreover, we have $\mathbb E[\widetilde Z^j]\ge0$ for $j=0,1,2$. Hence,~\eqref{Bernoulli lemma eq} and induction on $k$ yield that $\mathbb E[\widetilde Z^k]\ge0$ for all $k$. Since moreover $\mathbb E[\widetilde Z^j]>0$ for all even $j$, the right-hand side of \eqref{special Bernoulli lemma eq}  is strictly positive for $k\ge3$.\end{proof}

For the proof of Theorem~\ref{signed var thm}, it will be convenient to introduce the following notations. If $g\in C[0,1]$,  $n\in\mathbb{N}$, and $k\in\{0,\dots, b^n-1\}$, we write
$$\Delta_{n,k}g:=g((k+1)b^{-n})-g(kb^{-n}).
$$
Next, in analogy to~\eqref{Vptng} we define for $p\in\mathbb N$,
$$
\widehat V_{p,n}(g):=\sum_{k=0}^{b^n-1}(\Delta_{n,k}g)^p.
$$
Lemma~\ref{additive p-var lemma} does not work for signed $p^{\text{th}}$ variation. We will therefore need the following alternative argument. 

\begin{lemma}\label{signed additive p-var lemma}
Suppose  that $g\in C[0,1]$ is a function of bounded variation  and
 $p\in\{2,3,\dots\}$. If for some  $q>p$ a function $h\in C[0,1]$ has vanishing $q^{\text{th}}$ variation along the $b$-adic partitions \eqref{b-adic partitions},  then  the limit $\lim_n \widehat V_{p,n}(h)$ exists if and only if $\lim_n \widehat V_{p,n}(g+h)$ exists, and, in this case, both limits are equal.
\end{lemma}

\begin{proof}Applying Young's inequality with $q$ and $r:=q/(q-1)$ in the fourth step of the following estimate yields
\begin{align*}
\big|\widehat V_{p,n}(g+h)-\widehat V_{p,n}(h)\big|&\le\sum_{k=0}^{b^n-1}\big|(\Delta_{n,k}(g+h))^p-(\Delta_{n,k}h)^p\big|=\sum_{k=0}^{b^n-1}\bigg|\sum_{\ell=0}^{p-1}{p\choose \ell}(\Delta_{n,k}g)^{p-\ell}(\Delta_{n,k}h)^\ell\bigg|\\
&\le \sum_{k=0}^{b^n-1}|\Delta_{n,k}g|^{p}+ \sum_{k=0}^{b^n-1}\sum_{\ell=1}^{p-1}{p\choose \ell}|\Delta_{n,k}g|^{p-\ell}|\Delta_{n,k}h|^\ell\\
&\le  \sum_{k=0}^{b^n-1}|\Delta_{n,k}g|^{p}+\sum_{k=0}^{b^n-1}\sum_{\ell=1}^{p-1}{p\choose \ell}\bigg(\frac1r|\Delta_{n,k}g|^{(p-\ell)r}+\frac1q |\Delta_{n,k}h|^{\ell q}\bigg)\\
&=V_{n,1,p}(g) +\sum_{\ell=1}^{p-1}{p\choose \ell}\bigg(\frac1r V_{n,1,(p-\ell)r}(g)+\frac1qV_{n,1,\ell q}(h)\bigg),
\end{align*}
where we have used the notation~\eqref{Vptng} in the final step. It is easy to see that our assumptions on $g$ and $h$ imply that $V_{n,1,s}(g)\to0$ and $V_{n,1,\ell q}(h)\to0$ as $n\uparrow\infty$ for all $s>1$ and each $\ell\ge1$. This proves the assertion.
\end{proof}

\begin{proof}[Proof of Theorem~\ref{signed var thm}] Our assumption $|\alpha|>1/b$ clearly implies that the odd integer $q$ is larger than or equal to $3$. By dropping the absolute values in Lemma~\ref{lemma 1} and its proof, we see that for 
 $n\in\mathbb{N}$ and $(Y_m)_{m\in\mathbb N}$ as in~\eqref{R and Y},
 \begin{equation}\label{lemma subs eq}
 \sum_{k=0}^{b^n-1}\big(f((k+1)b^{-n})-f(kb^{-n})\big)^q= (\text{sgn}\,\alpha)^n\mathbb{E}\bigg[\Big(\sum_{m=1}^{n}(\alpha b)^{-m}Y_m\Big)^q\bigg].
 \end{equation}
 Proposition~\ref{skew prop} and Lemma~\ref{Bernoulli lemma} yield that the expectation on the right-hand side vanishes asymptotically  if and only if $b$ is even and $\ell=b/2$. In this case, the limit on the left-hand side of~\eqref{lemma subs eq}
 exists and is zero. In all other cases, the expectation on the right-hand side of~\eqref{lemma subs eq}
  will converge to $\mathbb E[Z^q]\neq0$, and  so~\eqref{signed qth variation} and~\eqref{nonexist signed qth variation}
 hold for $t=1$. The case of a general $t\in[0,1]$ then follows basically as in the proof of Theorem~\ref{thm}. One only needs to replace Lemma~\ref{additive p-var lemma} with Lemma~\ref{signed additive p-var lemma} and to consider the sequences of odd and even $n$ separately for $\alpha<0$.
\end{proof}

\parskip-0.5em\renewcommand{\baselinestretch}{0.9}\normalsize
\bibliography{CTbook}{}
\bibliographystyle{plain}
\end{document}